\documentclass[reqno]{amsart}

\usepackage[english]{babel}
\usepackage{amsthm, latexsym, enumerate, gastex, amssymb,url,listings}

\usepackage[mathcal]{eucal}
\usepackage{graphicx,url}

\copyrightinfo{2013}{Simone Ugolini}

\DeclareMathOperator{\ord}{ord}

\renewcommand{\phi}[0]{\varphi}
\renewcommand{\theta}[0]{\vartheta}
\renewcommand{\epsilon}[0]{\varepsilon}

\newcommand{\N}{\text{$\mathbf{N}$}}
\newcommand{\Z}{\text{$\mathbf{Z}$}}

\newtheorem{theorem}{Theorem}[section]
\newtheorem{lemma}[theorem]{Lemma}

\theoremstyle{definition}

\newtheorem{example}[theorem]{Example}

\theoremstyle{remark}

\numberwithin{equation}{section}

\begin{document}

\bibliographystyle{amsplain}

\date{}

\title[]
{A graph-based approach to repeating decimals}

\author{S.~Ugolini}
\email{sugolini@gmail.com} 

\begin{abstract}
In this paper we deal with a classical problem in elementary number theory, namely repeating decimals. We show how the digits of the period of the decimal representation of any fraction $\frac{k}{m}$, where $k$ and $m$ are positive integers arbitrarily chosen, can be obtained relying upon the graphs associated with the iteration of a certain map over the finite set $\{0, 1, \dots, 10n-2 \}$ for a suitable integer $n$, which depends on $m$. In the last section of the paper we generalize the results to any arbitrary choice of the base $B \geq 2$ for the representation of the fraction $\frac{k}{m}$. 
\end{abstract}

\maketitle

\section{Introduction}
Repeating decimals are a classical topic in elementary number theory. Many investigations can be found in the literature about such a topic. Among others we would like to mention \cite{kap} and \cite{rao}.
 
In this paper we introduce a graph-based procedure to get the digits of the period of the repeating decimal representations of the fractions
\begin{equation}\label{gen_frac}
\frac{k}{m},
\end{equation}
where $k$ and $m$ are arbitrarily chosen positive integers. In general we will always assume that $k < m$. Indeed, if $k \geq m$, then $\frac{k}{m} = k' + \frac{k''}{m}$ for some integers $k', k''$ with $1 \leq k'' < m$.

As we will notice in Sections \ref{km10m1} and \ref{km_arb}, it is sufficient to consider the fractions
\begin{equation}\label{k10n}
\frac{k}{10n-1},
\end{equation}
where $k$ and $n$ are positive integers such that $1 \leq k < 10n-1$. For this reason, in Section \ref{k10n1} we will just deal with fractions as (\ref{k10n}). The fractions as (\ref{k10n}) are quite manageable, since the sequence of the digits of the period of their decimal expansion is  equal to the sequence of the rightmost digits of the remainders (see Theorem  \ref{theo_proc}). 

In the final section of the paper we show how all our results read in a generic base $B \geq 2$ not necessarily equal to $10$.

\section{The graph-based approach to the fractions $\frac{k}{10n-1}$}\label{k10n1}
For any positive integer $n$ we define the set 
\begin{equation*}
R_n = \{0, 1, \dots, 10n-2 \}
\end{equation*}
formed by the $10n-1$ nonnegative integers smaller than $10n-1$. We notice that $R_n$  forms a set of representatives of the residue classes of $\Z / (10n-1) \Z$, namely $R_n$ is the set of all possible (nonnegative) remainders of the Euclidean division by $10n-1$ on the integers. 

Consider now the map
\begin{displaymath}
\begin{array}{rccc}
f_n : & R_n & \to & R_n\\
 & x & \mapsto & r,
\end{array}
\end{displaymath}
where $r$ is the remainder of the Euclidean division of $10x$ by $10n-1$, namely the nonnegative integer $r$ such that
\begin{displaymath}
\begin{cases}
10 x = q (10n-1) + r\\
0 \leq r < 10n-1,
\end{cases}
\end{displaymath}
for some integer $q$.

The following holds for the iterates of the map $f_n$.

\begin{lemma}\label{fnit}
If $i \in \N$ and $x \in R_n$, then $f_n^i(x)$ is equal to the integer $r_i$, being $r_i$ the remainder of the division of $10^ix$ by $10n-1$, namely
\begin{displaymath}
\begin{cases}
10^i  x = q_i (10n-1) + r_i\\
0 \leq r_i < 10n-1
\end{cases}
\end{displaymath}
for some integer $q_i$.
\end{lemma}
\begin{proof}
We prove the claim by induction on $i \in \N$.

If $i=0$, then $f_n^i (x) = x$. Since $10^0 x=x$ and
\begin{displaymath}
\begin{cases}
x = 0 \cdot (10n-1) + x\\
0 \leq x < 10n-1,
\end{cases}
\end{displaymath}
the assertion is proved. 

If $i > 0$, then $f_n^i (x) = f_n (f_n^{i-1} (x))$, where $f_n^{i-1} (x) = r_{i-1}$ by inductive hypothesis, namely
\begin{displaymath}
\begin{cases}
10^{i-1} x = q_{i-1} (10n-1) + r_{i-1}\\
0 \leq r_{i-1} < 10n-1
\end{cases}
\end{displaymath}
for some integer $q_{i-1}$.

Let $f_n(r_{i-1}) = r$. Then,
\begin{displaymath}
\begin{cases}
10 r_{i-1} = q (10n-1) + r\\
0 \leq r < 10n-1
\end{cases}
\end{displaymath}
for some integer $q$. 

Since $10 r_{i-1} = 10^i x - 10 q_{i-1} (10n-1)$, we have that
\begin{displaymath}
\begin{cases}
10^{i} x = (q+10q_{i-1}) (10n-1) + r\\
0 \leq r < 10n-1.
\end{cases}
\end{displaymath}
We remind that $f_n^i(x) = f_n(r_{i-1}) = r$. Hence, $f_n^i (x)$ is equal to the remainder of the division of $10^ix$ by $10n-1$ and we are done.\end{proof}

We can visualize the dynamics of the map $f_n$ on $R_n$ by means of a graph $G_n$, whose vertices are labelled by the elements of $R_n$ and where an arrow joins a vertex $i$ to a vertex $j$ provided that $j = f_n(i)$. Any element $i \in G_n$ is $f_n$-periodic, namely $f_n^s(i)=i$ for some positive integer $s$. In fact, $[10]$ is invertible in $\Z / (10n-1) \Z$. Therefore, the smallest among the positive integers $s$ such that $f_n^s(i)=i$ is $t$, where
\begin{displaymath}
t = \begin{cases}
1 & \text{if $i=0$,}\\
\ord_{d_i} ([10]) & \text{otherwise,}
\end{cases}
\end{displaymath}
being $\ord_{d_i} ([10])$ the multiplicative order of $[10]$ in $U(\Z / d_i \Z)$, where 
\begin{equation*}
d_i = \frac{10n-1}{\gcd(10n-1,i)}.
\end{equation*}

Let $\phi$ be the Euler's totient function. We remind that
\begin{equation*}
\sum_{\substack{d \in \N \\ d \mid (10n-1)}} \phi(d) = 10n-1.
\end{equation*}
For any positive divisor $d$ of $10n-1$ there are $\phi(d)$ elements $i \in R_n$ such that $d = \frac{10n-1}{\gcd(10n-1,i)}$. Such elements give rise to exactly $\frac{\phi(d)}{\ord_d ([10])}$ cycles of length $\ord_d ([10])$. 

\begin{example}\label{G_4}
In this example we construct the graph $G_{4}$, namely the graph associated with $f_4$ over the set $R_4$.

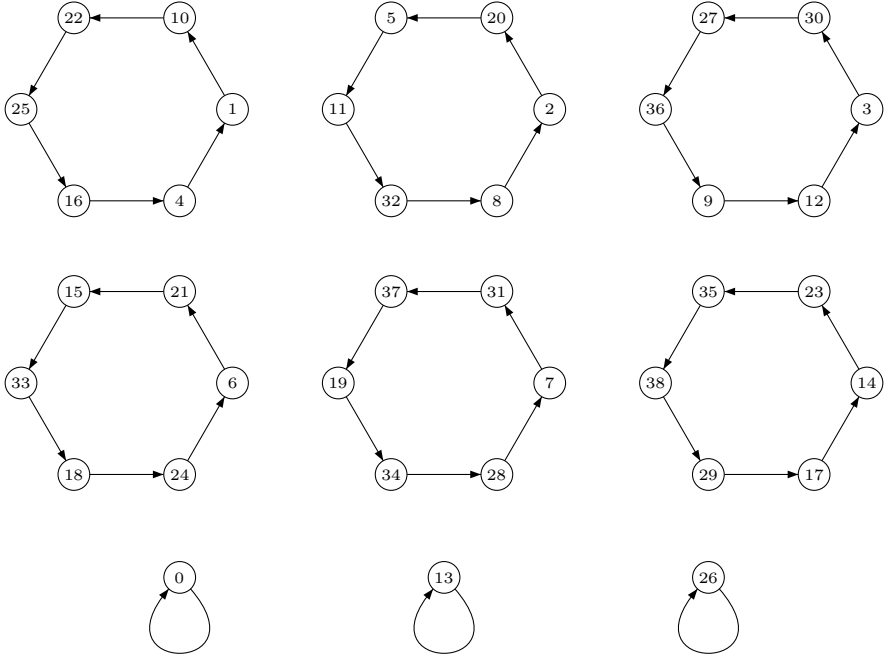
\begin{center}
    \unitlength=2pt
    \begin{picture}(160, 50)(-80,-25)
    \gasset{Nw=6,Nh=6,Nmr=3.0,curvedepth=0}
    \thinlines
    \tiny
  
    \node(N1)(-40,0){$1$}
    \node(N2)(-50,17.32){$10$}
    \node(N3)(-70,17.32){$22$}
    \node(N4)(-80,0){$25$}
    \node(N5)(-70,-17.32){$16$}
    \node(N6)(-50,-17.32){$4$}
    
    \node(N11)(20,0){$2$}
    \node(N12)(10,17.32){$20$}
    \node(N13)(-10,17.32){$5$}
    \node(N14)(-20,0){$11$}
    \node(N15)(-10,-17.32){$32$}
    \node(N16)(10,-17.32){$8$}
    
    \node(N21)(80,0){$3$}
    \node(N22)(70,17.32){$30$}
    \node(N23)(50,17.32){$27$}
    \node(N24)(40,0){$36$}
    \node(N25)(50,-17.32){$9$}
    \node(N26)(70,-17.32){$12$}

    \drawedge(N1,N2){}
    \drawedge(N2,N3){}
    \drawedge(N3,N4){}
    \drawedge(N4,N5){}
    \drawedge(N5,N6){}
    
    \drawedge(N11,N12){}
    \drawedge(N12,N13){}
    \drawedge(N13,N14){}
    \drawedge(N14,N15){}
    \drawedge(N15,N16){}
    
    \drawedge(N21,N22){}
    \drawedge(N22,N23){}
    \drawedge(N23,N24){}
    \drawedge(N24,N25){}
    \drawedge(N25,N26){}

    \drawedge(N6,N1){}
    \drawedge(N16,N11){}
    \drawedge(N26,N21){}

\end{picture}
\end{center}

\begin{center}
    \unitlength=2pt
    \begin{picture}(160, 50)(-80,-25)
    \gasset{Nw=6,Nh=6,Nmr=3.0,curvedepth=0}
    \thinlines
    \tiny
  
    \node(N1)(-40,0){$6$}
    \node(N2)(-50,17.32){$21$}
    \node(N3)(-70,17.32){$15$}
    \node(N4)(-80,0){$33$}
    \node(N5)(-70,-17.32){$18$}
    \node(N6)(-50,-17.32){$24$}
    
    \node(N11)(20,0){$7$}
    \node(N12)(10,17.32){$31$}
    \node(N13)(-10,17.32){$37$}
    \node(N14)(-20,0){$19$}
    \node(N15)(-10,-17.32){$34$}
    \node(N16)(10,-17.32){$28$}
    
    \node(N21)(80,0){$14$}
    \node(N22)(70,17.32){$23$}
    \node(N23)(50,17.32){$35$}
    \node(N24)(40,0){$38$}
    \node(N25)(50,-17.32){$29$}
    \node(N26)(70,-17.32){$17$}

    \drawedge(N1,N2){}
    \drawedge(N2,N3){}
    \drawedge(N3,N4){}
    \drawedge(N4,N5){}
    \drawedge(N5,N6){}
    
    \drawedge(N11,N12){}
    \drawedge(N12,N13){}
    \drawedge(N13,N14){}
    \drawedge(N14,N15){}
    \drawedge(N15,N16){}
    
    \drawedge(N21,N22){}
    \drawedge(N22,N23){}
    \drawedge(N23,N24){}
    \drawedge(N24,N25){}
    \drawedge(N25,N26){}

    \drawedge(N6,N1){}
    \drawedge(N16,N11){}
    \drawedge(N26,N21){}

\end{picture}
\end{center}

\begin{center}
    \unitlength=2pt
    \begin{picture}(120, 30)(0,-20)
    \gasset{Nw=6,Nh=6,Nmr=3.0,curvedepth=0}
    \thinlines
    \tiny
  
    \node(N61)(10,0){$0$}
    \node(N71)(60,0){$13$}
    \node(N81)(110,0){$26$}
             
	\drawloop[loopangle=-90](N61){}
	\drawloop[loopangle=-90](N71){}
	\drawloop[loopangle=-90](N81){}
\end{picture}
\end{center}
We notice that in $R_4$ there are $39$ elements. The set of positive integer divisors of $39$ is $\{1, 3, 13, 39 \}$. In the following table we summarize some relevant data for our graph.
\begin{center}
\begin{tabular}{|c|c|c|c|}
\hline 
$d$ & $\ord_d ([10])$ & $\phi(d)$ & $\phi(d) / \ord_d ([10])$ \\ 
\hline 
3 & 1 & 2 & 2\\ 
\hline 
13 & 6 & 12 & 2 \\ 
\hline 
39 & 6 & 24 & 4\\ 
\hline  
\end{tabular} 
\end{center}
From the table we deduce the following.
\begin{itemize}
\item There are $2$ elements $i \in R_4$ such that $d = \frac{39}{\gcd(39, i)} = 3$. Such elements give rise  to $2$ cycles of length $1$ each.
\item There are $12$ elements $i \in R_4$ such that $d = \frac{39}{\gcd(39, i)} = 13$. Such elements give rise  to $2$ cycles of length $6$ each.
\item There are $24$ elements $i \in R_4$ such that $d = \frac{39}{\gcd(39, i)} = 39$. Such elements give rise  to $4$ cycles of length $6$ each.
\end{itemize}
Finally, if $d=1$, then there is just $1$ element $i \in R_4$ such that $d = \frac{39}{\gcd(39,i)} = 1$, namely $i=0$. Such element gives rise to $1$ cycle of length $1$. 
\end{example}

\begin{example}\label{G_12}
In this example we construct the graph $G_{12}$, namely the graph associated with $f_{12}$ over the set $R_{12}$.

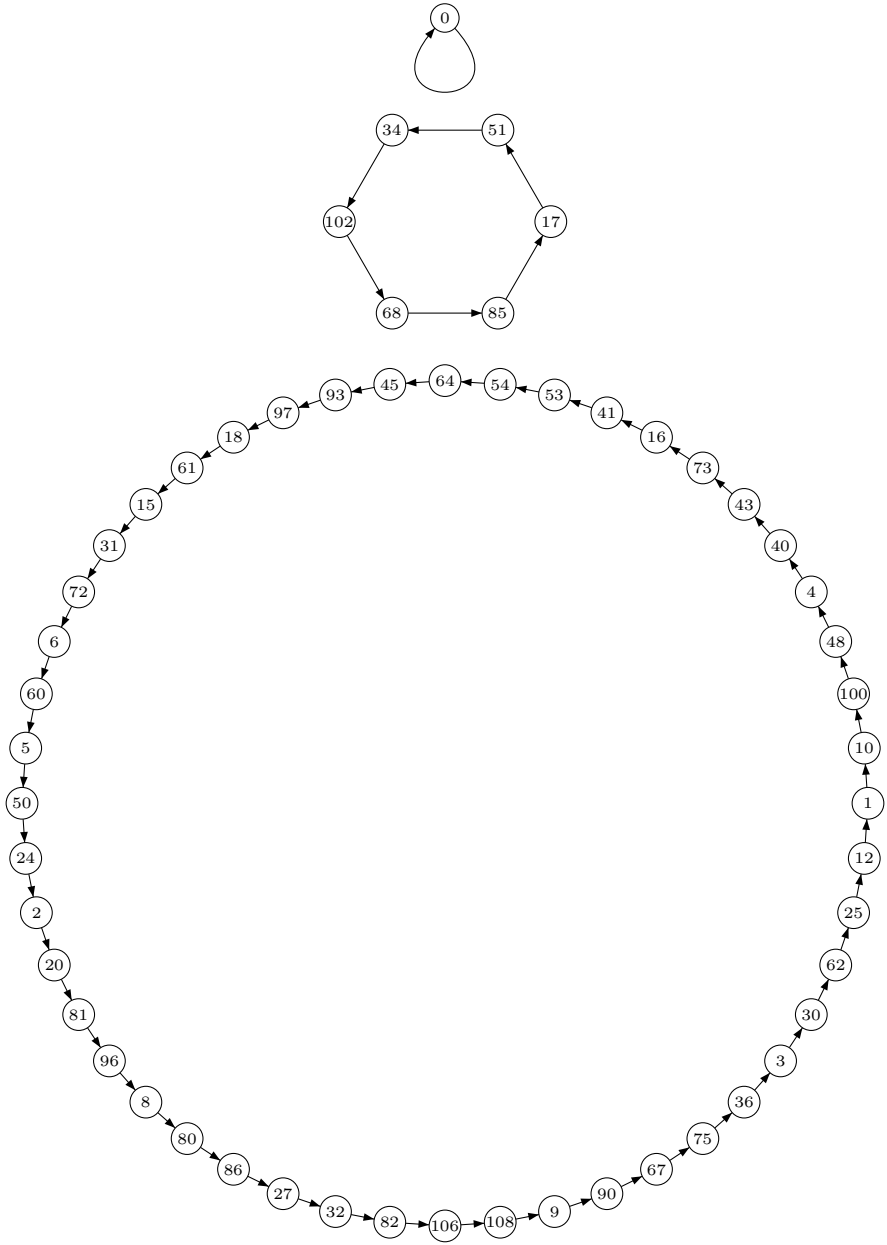
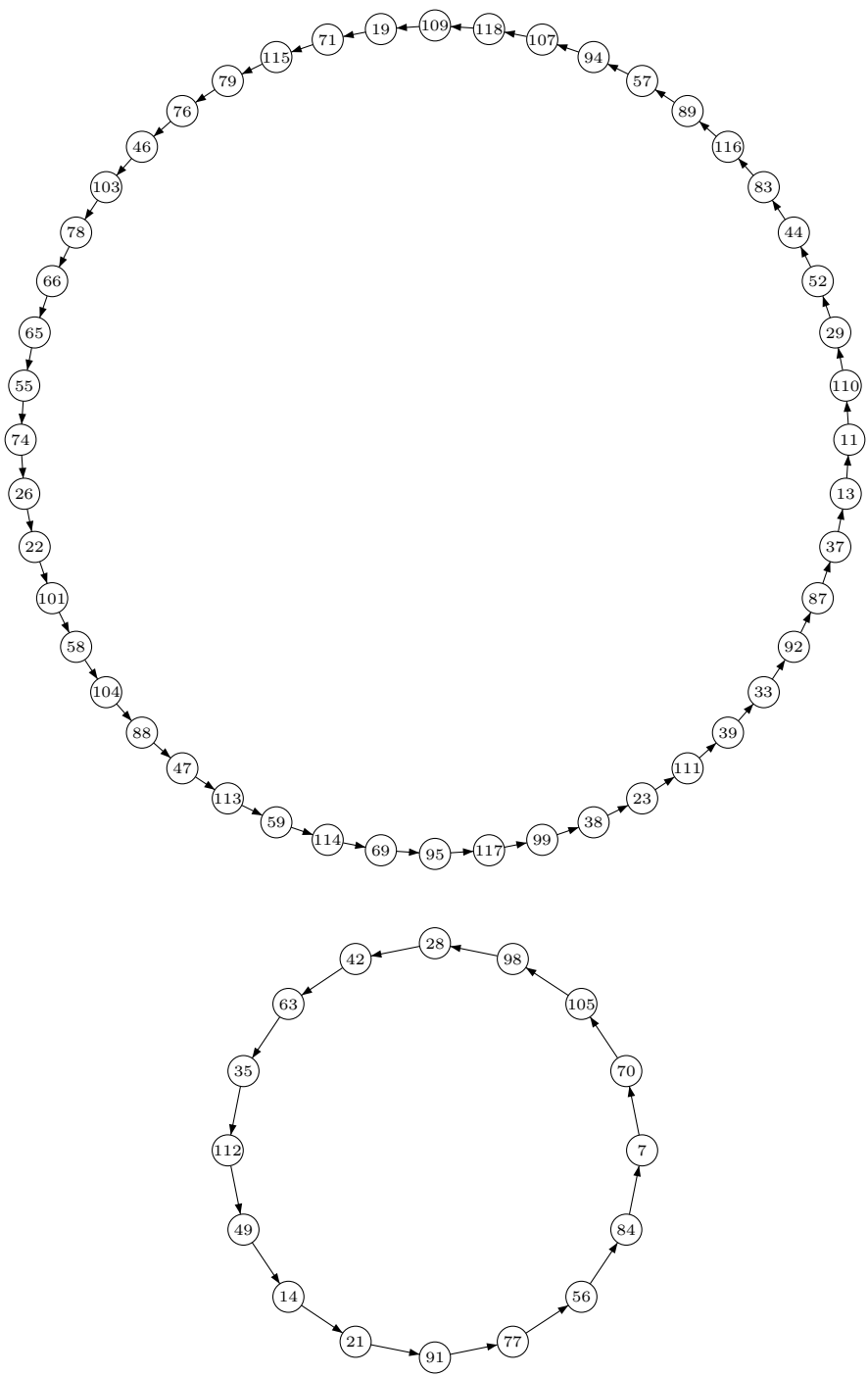
\begin{center}
    \unitlength=1.8pt
    \begin{picture}(160, 20)(0,-15)
    \gasset{Nw=6,Nh=6,Nmr=3.0,curvedepth=0}
    \thinlines
    \tiny
	\node(C1)(80,0){$0$}

    \gasset{curvedepth=10.0}

	\drawloop[loopangle=-90](C1){}
	\end{picture}
\end{center}

\begin{center}
    \unitlength=2pt
    \begin{picture}(160, 40)(-40,-20)
    \gasset{Nw=6,Nh=6,Nmr=3.0,curvedepth=0}
    \thinlines
    \tiny
       
    \node(B1)(60,0){$17$}
    \node(B2)(50,17.32){$51$}
    \node(B3)(30,17.32){$34$}
    \node(B4)(20,0){$102$}
    \node(B5)(30,-17.32){$68$}
    \node(B6)(50,-17.32){$85$}

    \drawedge(B1,B2){}
    \drawedge(B2,B3){}
    \drawedge(B3,B4){}
    \drawedge(B4,B5){}
    \drawedge(B5,B6){}
    \drawedge(B6,B1){}

\end{picture}
\end{center}
	
\begin{center}
    \unitlength=2pt
    \begin{picture}(160, 170)(-80,-85)
    \gasset{Nw=6,Nh=6,Nmr=3.0,curvedepth=0}
    \thinlines
    \tiny
  
    \node(A1)(80,0){$1$}
    \node(A2)(79.3,10.44){$10$}
    \node(A3)(77.28,20.71){$100$}
    \node(A4)(73.9,30.62){$48$}
    \node(A5)(69.28,40){$4$}
    \node(A6)(63.47,48.77){$40$}
    \node(A7)(56.57,56.57){$43$}
    \node(A8)(48.77,63.47){$73$}
    \node(A9)(40,69.28){$16$}
    \node(A10)(30.62,73.9){$41$}
    \node(A11)(20.71,77.28){$53$}
    \node(A12)(10.44,79.3){$54$}
    \node(A24)(-79.3,10.44){$5$}
    \node(A23)(-77.28,20.71){$60$}
    \node(A22)(-73.9,30.62){$6$}
    \node(A21)(-69.28,40){$72$}
    \node(A20)(-63.47,48.77){$31$}
    \node(A19)(-56.57,56.57){$15$}
    \node(A18)(-48.77,63.47){$61$}
    \node(A17)(-40,69.28){$18$}
    \node(A16)(-30.62,73.9){$97$}
    \node(A15)(-20.71,77.28){$93$}
    \node(A14)(-10.44,79.3){$45$}
    \node(A13)(0,80){$64$}
    \node(A25)(-80,0){$50$}
    \node(A26)(-79.3,-10.44){$24$}
    \node(A27)(-77.28,-20.71){$2$}
    \node(A28)(-73.9,-30.62){$20$}
    \node(A29)(-69.28,-40){$81$}
    \node(A30)(-63.47,-48.77){$96$}
    \node(A31)(-56.57,-56.57){$8$}
    \node(A32)(-48.77,-63.47){$80$}
    \node(A33)(-40,-69.28){$86$}
    \node(A34)(-30.62,-73.9){$27$}
    \node(A35)(-20.71,-77.28){$32$}
    \node(A36)(-10.44,-79.3){$82$}
    \node(A48)(79.3,-10.44){$12$}
    \node(A47)(77.28,-20.71){$25$}
    \node(A46)(73.9,-30.62){$62$}
    \node(A45)(69.28,-40){$30$}
    \node(A44)(63.47,-48.77){$3$}
    \node(A43)(56.57,-56.57){$36$}
    \node(A42)(48.77,-63.47){$75$}
    \node(A41)(40,-69.28){$67$}
    \node(A40)(30.62,-73.9){$90$}
    \node(A39)(20.71,-77.28){$9$}
    \node(A38)(10.44,-79.3){$108$}
    \node(A37)(0,-80){$106$}

    \drawedge(A1,A2){}
    \drawedge(A2,A3){}
    \drawedge(A3,A4){}
    \drawedge(A4,A5){}
    \drawedge(A5,A6){}
    \drawedge(A6,A7){}
    \drawedge(A7,A8){}
    \drawedge(A8,A9){}
    \drawedge(A9,A10){}
    \drawedge(A10,A11){}
    \drawedge(A11,A12){}
    \drawedge(A12,A13){}
    \drawedge(A13,A14){}
    \drawedge(A14,A15){}
    \drawedge(A15,A16){}
    \drawedge(A16,A17){}
    \drawedge(A17,A18){}
    \drawedge(A18,A19){}
    \drawedge(A19,A20){}
    \drawedge(A20,A21){}
    \drawedge(A21,A22){}
    \drawedge(A22,A23){}
    \drawedge(A23,A24){}
    \drawedge(A24,A25){}
    \drawedge(A25,A26){}
    \drawedge(A26,A27){}
    \drawedge(A27,A28){}
    \drawedge(A28,A29){}
    \drawedge(A29,A30){}
    \drawedge(A30,A31){}
    \drawedge(A31,A32){}
    \drawedge(A32,A33){}
    \drawedge(A33,A34){}
    \drawedge(A34,A35){}
    \drawedge(A35,A36){}
    \drawedge(A36,A37){}
    \drawedge(A37,A38){}
    \drawedge(A38,A39){}
    \drawedge(A39,A40){}
    \drawedge(A40,A41){}
    \drawedge(A41,A42){}
    \drawedge(A42,A43){}
    \drawedge(A43,A44){}
    \drawedge(A44,A45){}
    \drawedge(A45,A46){}
    \drawedge(A46,A47){}
    \drawedge(A47,A48){}
    \drawedge(A48,A1){}
\end{picture}
\end{center}

\begin{center}
    \unitlength=2pt
    \begin{picture}(160, 170)(-80,-85)
    \gasset{Nw=6,Nh=6,Nmr=3.0,curvedepth=0}
    \thinlines
    \tiny
  
    \node(A1)(80,0){$11$}
    \node(A2)(79.3,10.44){$110$}
    \node(A3)(77.28,20.71){$29$}
    \node(A4)(73.9,30.62){$52$}
    \node(A5)(69.28,40){$44$}
    \node(A6)(63.47,48.77){$83$}
    \node(A7)(56.57,56.57){$116$}
    \node(A8)(48.77,63.47){$89$}
    \node(A9)(40,69.28){$57$}
    \node(A10)(30.62,73.9){$94$}
    \node(A11)(20.71,77.28){$107$}
    \node(A12)(10.44,79.3){$118$}
    \node(A24)(-79.3,10.44){$55$}
    \node(A23)(-77.28,20.71){$65$}
    \node(A22)(-73.9,30.62){$66$}
    \node(A21)(-69.28,40){$78$}
    \node(A20)(-63.47,48.77){$103$}
    \node(A19)(-56.57,56.57){$46$}
    \node(A18)(-48.77,63.47){$76$}
    \node(A17)(-40,69.28){$79$}
    \node(A16)(-30.62,73.9){$115$}
    \node(A15)(-20.71,77.28){$71$}
    \node(A14)(-10.44,79.3){$19$}
    \node(A13)(0,80){$109$}
    \node(A25)(-80,0){$74$}
    \node(A26)(-79.3,-10.44){$26$}
    \node(A27)(-77.28,-20.71){$22$}
    \node(A28)(-73.9,-30.62){$101$}
    \node(A29)(-69.28,-40){$58$}
    \node(A30)(-63.47,-48.77){$104$}
    \node(A31)(-56.57,-56.57){$88$}
    \node(A32)(-48.77,-63.47){$47$}
    \node(A33)(-40,-69.28){$113$}
    \node(A34)(-30.62,-73.9){$59$}
    \node(A35)(-20.71,-77.28){$114$}
    \node(A36)(-10.44,-79.3){$69$}
    \node(A48)(79.3,-10.44){$13$}
    \node(A47)(77.28,-20.71){$37$}
    \node(A46)(73.9,-30.62){$87$}
    \node(A45)(69.28,-40){$92$}
    \node(A44)(63.47,-48.77){$33$}
    \node(A43)(56.57,-56.57){$39$}
    \node(A42)(48.77,-63.47){$111$}
    \node(A41)(40,-69.28){$23$}
    \node(A40)(30.62,-73.9){$38$}
    \node(A39)(20.71,-77.28){$99$}
    \node(A38)(10.44,-79.3){$117$}
    \node(A37)(0,-80){$95$}

    \drawedge(A1,A2){}
    \drawedge(A2,A3){}
    \drawedge(A3,A4){}
    \drawedge(A4,A5){}
    \drawedge(A5,A6){}
    \drawedge(A6,A7){}
    \drawedge(A7,A8){}
    \drawedge(A8,A9){}
    \drawedge(A9,A10){}
    \drawedge(A10,A11){}
    \drawedge(A11,A12){}
    \drawedge(A12,A13){}
    \drawedge(A13,A14){}
    \drawedge(A14,A15){}
    \drawedge(A15,A16){}
    \drawedge(A16,A17){}
    \drawedge(A17,A18){}
    \drawedge(A18,A19){}
    \drawedge(A19,A20){}
    \drawedge(A20,A21){}
    \drawedge(A21,A22){}
    \drawedge(A22,A23){}
    \drawedge(A23,A24){}
    \drawedge(A24,A25){}
    \drawedge(A25,A26){}
    \drawedge(A26,A27){}
    \drawedge(A27,A28){}
    \drawedge(A28,A29){}
    \drawedge(A29,A30){}
    \drawedge(A30,A31){}
    \drawedge(A31,A32){}
    \drawedge(A32,A33){}
    \drawedge(A33,A34){}
    \drawedge(A34,A35){}
    \drawedge(A35,A36){}
    \drawedge(A36,A37){}
    \drawedge(A37,A38){}
    \drawedge(A38,A39){}
    \drawedge(A39,A40){}
    \drawedge(A40,A41){}
    \drawedge(A41,A42){}
    \drawedge(A42,A43){}
    \drawedge(A43,A44){}
    \drawedge(A44,A45){}
    \drawedge(A45,A46){}
    \drawedge(A46,A47){}
    \drawedge(A47,A48){}
    \drawedge(A48,A1){}
\end{picture}
\end{center}

\begin{center}
    \unitlength=2pt
    \begin{picture}(160, 100)(-120,-50)
    \gasset{Nw=6,Nh=6,Nmr=3.0,curvedepth=0}
    \thinlines
    \tiny
  
    \node(A1)(0,0){$7$}
    \node(A2)(-3.05,15.31){$70$}
    \node(A3)(-11.72,28.28){$105$}
    \node(A4)(-24.96,36.96){$98$}
    \node(A5)(-40,40){$28$}
    \node(A6)(-55.31,36.96){$42$}
    \node(A7)(-68.28,28.28){$63$}
    \node(A8)(-76.96,15.31){$35$}
    \node(A9)(-80,0){$112$}
    \node(A10)(-76.96,-15.31){$49$}
    \node(A11)(-68.28,-28.28){$14$}
    \node(A12)(-55.31,-36.96){$21$}
    \node(A13)(-40,-40){$91$}
    \node(A14)(-24.96,-36.96){$77$}
    \node(A15)(-11.72,-28.28){$56$}
    \node(A16)(-3.05,-15.31){$84$}

    \drawedge(A1,A2){}
    \drawedge(A2,A3){}
    \drawedge(A3,A4){}
    \drawedge(A4,A5){}
    \drawedge(A5,A6){}
    \drawedge(A6,A7){}
    \drawedge(A7,A8){}
    \drawedge(A8,A9){}
    \drawedge(A9,A10){}
    \drawedge(A10,A11){}
    \drawedge(A11,A12){}
    \drawedge(A12,A13){}
    \drawedge(A13,A14){}
    \drawedge(A14,A15){}
    \drawedge(A15,A16){}
    \drawedge(A16,A1){}

\end{picture}
\end{center}

We notice that in $R_{12}$ there are $119$ elements. The set of positive integer divisors of $119$ is $\{1, 7, 17, 119 \}$. In the following table we summarize some relevant data for our graph.
\begin{center}
\begin{tabular}{|c|c|c|c|}
\hline 
$d$ & $\ord_d ([10])$ & $\phi(d)$ & $\phi(d) / \ord_d ([10])$ \\ 
\hline 
7 & 6 & 6 & 1\\ 
\hline 
17 & 16 & 16 & 1 \\ 
\hline 
119 & 48 & 96 & 2\\ 
\hline  
\end{tabular} 
\end{center}
From the table we deduce the following.
\begin{itemize}
\item There are $6$ elements $i \in R_{12}$ such that $d = \frac{119}{\gcd(119, i)} = 7$. Such elements give rise  to $1$ cycle of length $6$.
\item There are $16$ elements $i \in R_{12}$ such that $d = \frac{119}{\gcd(119, i)} = 17$. Such elements give rise  to $1$ cycle of length $6$.
\item There are $96$ elements $i \in R_{12}$ such that $d = \frac{119}{\gcd(119, i)} = 119$. Such elements give rise  to $2$ cycles of length $48$ each.
\end{itemize}
Finally, if $d=1$, then there is just $1$ element $i \in R_{119}$ such that $d = \frac{119}{\gcd(119,i)} = 1$, namely $i=0$. Such element gives rise to $1$ cycle of length $1$. 
\end{example}

Consider now the map
\begin{equation*}
u : \N \to \{0, \dots, 9 \},
\end{equation*}
which takes any nonnegative integer $d$, represented in base $10$, to its rightmost digit. For example, $u(127) = 7$.
The following properties hold for the map $u$.
\begin{lemma}\label{uprop}
Let $a$ and $b$ be two positive integers. Then:
\begin{itemize}
\item $u(a+b) = u (u(a)+u(b))$;
\item $u(a \cdot b) = u (u(a) \cdot u(b))$.
\end{itemize}
\end{lemma}
\begin{proof}
If we perform the Euclidean division of $a$ and $b$ by $10$, then there exist four integers $a_0, a_1, b_0$ and $b_1$ such that
\begin{displaymath}
\begin{cases}
a = a_1 \cdot 10 + a_0,\\
0 \leq a_0 \leq 9,
\end{cases}
\quad
\begin{cases}
b = b_1 \cdot 10 + b_0,\\
0 \leq b_0 \leq 9.
\end{cases}
\end{displaymath}
Hence, $u(a) = a_0$, $u(b)=b_0$ and  $a+b=(a_1+b_1) \cdot 10 + (a_0+b_0)$. Since $u((a_1+b_1) \cdot 10) = 0$, we have that the integer $(a_1+b_1) \cdot 10$ has no influence on the last digit of $a+b$. Therefore, $u(a+b) = u(a_0+b_0) = u(u(a) +u(b))$.

As regards the second property, we have that 
\begin{equation*}
u(a \cdot b) = u((a_1 b_0 + a_0 b_1 + a_1 b_1 \cdot 10) \cdot 10 + a_0 b_0).
\end{equation*}
Since $u((a_1 b_0 + a_0 b_1 + a_1 b_1 \cdot 10) \cdot 10) = 0$, we have that $u(a \cdot b) = u(a_0 \cdot b_0)$. Since $a_0 = u(a)$ and $b_0 = u(b)$, we get the result.
\end{proof}

We are now in a position to prove the following theorem.

\begin{theorem}\label{theo_proc}
Let $k$ and $n$ be two positive integers such that $1 \leq k < 10n-1$ and suppose that $ \frac{k}{10n-1} = 0. \overline{a_1 \dots a_l}$, where $l$ is the length of the period of the decimal expansion $0. \overline{a_1 \dots a_l}$. 

For any $i \in \{1, \dots, l \}$ define the positive integers $\tilde{q}_i, \tilde{r}_i$ such that
\begin{displaymath}
\begin{cases}
10^i \cdot k = \tilde{q}_i (10n-1) + \tilde{r}_i\\
0 \leq \tilde{r}_i < 10n-1.
\end{cases}
\end{displaymath}
Then, for any $i \in  \{1, \dots, l \}$:
\begin{itemize}
\item $u(\tilde{q}_i) = a_i$;
\item $u(\tilde{q}_i) = u(\tilde{r}_i)$.
\end{itemize} 
\end{theorem}
\begin{proof}
For any positive integer $i$ we have that
\begin{equation*}
10^i \cdot \frac{k}{10n-1} = a_1 \dots a_i . \overline{a_{i+i} \dots a_l a_{1} \dots a_{i}}.
\end{equation*} 
Since $10^i \cdot k = \tilde{q}_i (10n-1) + \tilde{r}_i$, we have that
\begin{equation*}
\frac{10^i k}{10n-1} = \frac{\tilde{q}_i (10n-1) + \tilde{r}_i}{10n-1} = \tilde{q}_i + \frac{\tilde{r}_i}{10n-1}.
\end{equation*}
By the fact that $0 \leq \tilde{r}_i < 10n-1$ we deduce that $\frac{\tilde{r}_i}{10n-1} < 1$. Hence, 
\begin{equation*}
\tilde{q}_i = a_1 \dots a_i
\end{equation*}
and $u(\tilde{q}_i) = a_i$.

As regards the second statement, since $u(10^i \cdot k) = u ( \tilde{q}_i (10n-1) + \tilde{r}_i )$ for any $i \in  \{1, \dots, l \}$, we get that $u ( \tilde{q}_i (10n-1) + \tilde{r}_i ) = 0$. 

By the properties of the map $u$ (see Lemma \ref{uprop}) we get the following:
\begin{eqnarray*}
u ( \tilde{q}_i (10n-1) + \tilde{r}_i ) & = & u ( u(\tilde{q}_i) u(10n-1) + u(\tilde{r}_i) ) =  u ( u(\tilde{q}_i) \cdot 9 + u(\tilde{r}_i) )\\
& = & u( u(\tilde{q}_i) \cdot (10-1) + u(\tilde{r}_i))  = u(u(\tilde{q}_i) \cdot 10 - u(\tilde{q}_i) + u(\tilde{r}_i))\\
& = & u(- u(\tilde{q}_i) + u(\tilde{r}_i)).
\end{eqnarray*}
Since 
\begin{equation*}
-9 \leq - u(\tilde{q}_i) + u(\tilde{r}_i) \leq 9,
\end{equation*}
we have that $u(- u(\tilde{q}_i) + u(\tilde{r}_i)) = 0$ if and only if $u(\tilde{q}_i) = u(\tilde{r}_i)$.
\end{proof}

As a consequence of Theorem \ref{theo_proc}, we can retrieve the digits of the period of the decimal expansion of $\frac{k}{10 \cdot n -1}$ from the graph $G_n$ as follows. We consider the cycle of $G_n$ containing the vertex $k$. Then, for any $1 \leq i \leq l$, the digit $a_i$ is equal to the rightmost digit of the $i$-th vertex following $k$ in the cycle.  
\begin{example}\label{ex_1_39}
If we want to compute the digits of the period of the decimal representation of $\frac{1}{39}$ we can rely on the graph $G_4$ (see Example \ref{G_4}). We look for the vertex $1$ in $G_4$ and then we visit all the successors of $1$ in the cycle containing $1$. In doing that we take note of the rightmost digit of any vertex we visit. When we reach the vertex $1$ we take note of $1$ and stop the procedure. In the following picture the bold digits are the aforementioned digits, which, read in the correct order, form the digits of the period of $\frac{1}{39} = 0.\overline{025641}$.
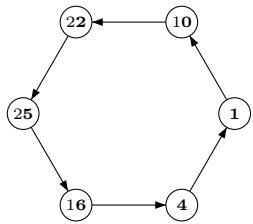
\begin{center}
    \unitlength=2pt
    \begin{picture}(160, 50)(-140,-25)
    \gasset{Nw=6,Nh=6,Nmr=3.0,curvedepth=0}
    \thinlines
    \tiny
  
    \node(N1)(-40,0){$\mathbf{1}$}
    \node(N2)(-50,17.32){$1 \mathbf{0}$}
    \node(N3)(-70,17.32){$2 \mathbf{2}$}
    \node(N4)(-80,0){$2 \mathbf{5}$}
    \node(N5)(-70,-17.32){$1 \mathbf{6}$}
    \node(N6)(-50,-17.32){$\mathbf{4}$}

    \drawedge(N1,N2){}
    \drawedge(N2,N3){}
    \drawedge(N3,N4){}
    \drawedge(N4,N5){}
    \drawedge(N5,N6){}   
    \drawedge(N6,N1){}
\end{picture}
\end{center}
Actually, by means of the same cycle, we can find the decimal representation of the following fractions
\begin{equation*}
\frac{10}{39}, \frac{22}{39}, \frac{25}{39}, \frac{16}{39}, \frac{4}{39}.
\end{equation*}  
For example, the decimal representation of $\frac{25}{39}$ can be retrieved starting from the vertex $25$. Indeed, $\frac{25}{39} = 0, \overline{641025}$.

\end{example}

\subsection{Reverting the edges}
Starting from any graph $G_n$ we can easily construct a new graph, which we call $G_n'$, obtained from $G_n$ reverting the edges. In doing that we construct a new graph, from which we can read off the digits of the period of any $\frac{k}{10n-1}$, for $1 \leq k < 10n-1$, from the right to the left. This is due to the fact that we are visiting the vertices of any cycle in the reverted order. Therefore, if we start from a vertex $k$, then the $i$-th successor vertex of $k$ in $G_n'$ is the remainder of the division of $n^i k$ by $10n-1$, since $[n]$ is the inverse of $[10]$ in $U(\Z / (10n-1) \Z)$. 

Indeed, if $l$ is the multiplicative order of $[10]$ in $U(\Z / (10n-1) \Z)$, then 
\begin{equation*}
[n^i] = [10^{l-i}]
\end{equation*}
for any integer $i$. Suppose that
\begin{displaymath}
\begin{cases}
10^{l-i} k = q_{l-i} (10n-1) + r_{l-i},\\
0 \leq r_{l-i} < 10n-1,
\end{cases}
\end{displaymath}
for some integers $q_{l-i}$ and $r_{l-i}$. Then,
\begin{equation*}
[n^i k] = [10^{l-i} k] = [r_{l-i}].
\end{equation*}
Therefore, $r_{l-i}$ is the remainder of the Euclidean division of $n^i k$ by $10n-1$.

\begin{example}
If we want to compute the digits of the period of the decimal representation of $\frac{1}{39}$, starting from the rightmost one, then we consider the cycle containing $1$ in $G_4'$. Such a cycle is easily constructed, once one knows $G_4$.
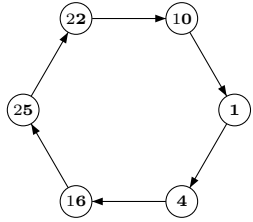
\begin{center}
    \unitlength=2pt
    \begin{picture}(160, 50)(-140,-25)
    \gasset{Nw=6,Nh=6,Nmr=3.0,curvedepth=0}
    \thinlines
    \tiny
  
    \node(N1)(-40,0){$\mathbf{1}$}
    \node(N2)(-50,17.32){$1 \mathbf{0}$}
    \node(N3)(-70,17.32){$2 \mathbf{2}$}
    \node(N4)(-80,0){$2 \mathbf{5}$}
    \node(N5)(-70,-17.32){$1 \mathbf{6}$}
    \node(N6)(-50,-17.32){$\mathbf{4}$}

    \drawedge(N1,N6){}
    \drawedge(N6,N5){}
    \drawedge(N5,N4){}
    \drawedge(N4,N3){}
    \drawedge(N3,N2){}   
    \drawedge(N2,N1){}
\end{picture}
\end{center}
Starting from $1$ we visit, in the order, $4$, $16$, $25$, $22$ and $10$. The sequence of the rightmost digits of such integers, whose first element is $1$ and the last element is $10$,  forms the sequence of the digits of the period of $\frac{1}{39}$ from the right to the left. 
\end{example}

\section{On the fractions $\frac{k}{m}$ with $\gcd(10,m)=1$}\label{km10m1}
We can easily retrieve the digits of the period of a fraction $\frac{k}{m}$, where $\gcd (10,m) =1$ and $1 \leq k < m$, from a fraction having $10  n -1$ as denominator, for some positive integer $n$. Indeed, $u(m) \in \{1, 3, 7, 9 \}$. Therefore, if $u(m)=9$, then $\frac{k}{m} = \frac{k}{10n-1}$ for some positive integer $n$. If, on the contrary, $u(m) \in \{1, 3, 7 \}$, then $\frac{k}{m} = \frac{k'}{m'}$, where 
\begin{eqnarray*}
k' & = & c \cdot k,\\
m' & = & c \cdot m, 
\end{eqnarray*}
being $c$ a positive integer defined as follows:
\begin{itemize}
\item if $u(m) = 1$, then $c=9$;
\item if $u(m) = 3$, then $c=3$;
\item if $u(m)=7$, then $c=7$.
\end{itemize}

\begin{example}
If we want to compute the digits of the period of the decimal representation of $\frac{1}{13}$, then we can rely on the graph $G_4$ (see Example \ref{G_4}). Indeed,  $\frac{1}{13} = \frac{3 \cdot 1}{3 \cdot 13} = \frac{3}{39}$. Therefore, we look at the cycle containing $3$ in $G_4$ and we deduce that $\frac{1}{13} = 0, \overline{076923}$.
\end{example}

\begin{example}
If we want to compute the digits of the period of the decimal representation of $\frac{1}{17}$, then we can rely on the graph $G_{12}$ (see Example \ref{G_12}). Indeed,  $\frac{1}{17} = \frac{7 \cdot 1}{7 \cdot 17} = \frac{7}{119}$. Therefore, we look at the cycle containing $7$ in $G_{12}$ and we deduce that $\frac{1}{17} = 0, \overline{0588235294117647}$.
\end{example}

\section{On the fractions $\frac{k}{m}$ where $m$ is arbitrarily chosen}\label{km_arb}
Let $k$ and $m$ be two positive integers such that $1 \leq k < m$ and
\begin{equation*}
m = 2^a \cdot 5^b \cdot m'
\end{equation*}
for some integer $m'$ coprime to $10$ and some nonnegative integers $a$ and $b$. Define $e = \max (a,b)$. Then,
\begin{equation*}
\frac{k}{m} = \frac{2^{e-a} \cdot 5^{e-b} \cdot k}{10^e \cdot m'} = \frac{1}{10^e} \cdot \frac{2^{e-a} \cdot 5^{e-b} \cdot k}{m'}.
\end{equation*}
If necessary we divide $2^{e-a} \cdot 5^{e-b} \cdot k$ by $m'$ and finally we get
\begin{equation*}
\frac{2^{e-a} \cdot 5^{e-b} \cdot k}{m'} = k' + \frac{k''}{m'},
\end{equation*}
where $k', k'' \in \N, \gcd (m', 10) =1$ and $1 \leq k'' < m'$.

\section{The generalization to the base-$B$ representation}
Let $B$ be an integer such that $B \geq 2$. 

For any positive integer $n$ we define the set
\begin{equation*}
R_{B,n} = \{0, 1, \dots, Bn-2 \}.
\end{equation*}

Consider now the map 
\begin{displaymath}
\begin{array}{rccc}
f_{B,n} : & R_{B,n} & \to & R_{B,n}\\
 & x & \mapsto & r,
\end{array}
\end{displaymath}
where $r$ is the remainder of the Euclidean division of $Bx$ by $Bn-1$, namely the nonnegative integer $r$ such that
\begin{displaymath}
\begin{cases}
B x = q (Bn-1) + r\\
0 \leq r < Bn-1,
\end{cases}
\end{displaymath}
for some integer $q$.

The following holds for the iterates of the map $f_{B,n}$.

\begin{lemma}\label{fBnit}
If $i \in \N$ and $x \in R_{B,n}$, then $f_{B,n}^i(x)$ is equal to the integer $r_i$, being $r_i$ the remainder of the division of $B^ix$ by $Bn-1$, namely
\begin{displaymath}
\begin{cases}
B^i  x = q_i (Bn-1) + r_i\\
0 \leq r_i < Bn-1
\end{cases}
\end{displaymath}
for some integer $q_i$.
\end{lemma}
\begin{proof}
The proof follows the same lines as the proof of Lemma \ref{fnit}, replacing all the occurrences of $10$ by $B$.
\end{proof}

In analogy with Section \ref{k10n1} we can visualize the dynamics of the map $f_{B,n}$ on $R_{B,n}$ by means of a graph $G_{B,n}$, whose vertices are labelled by the elements of $R_{B,n}$ and where an arrow joins a vertex $i$ to a vertex $j$ provided that $j = f_{B,n}(i)$. For any positive divisor $d$ of $Bn-1$ there are $\phi(d)$ elements $i \in R_{B,n}$ such that $d = \frac{Bn-1}{\gcd(Bn-1,i)}$. Such elements give rise to exactly $\frac{\phi(d)}{\ord_d ([B])}$ cycles of length $\ord_d ([B])$. 

Consider now the map
\begin{equation*}
u_B : \N \to \{0, \dots, B-1 \}
\end{equation*}
which takes any nonnegative integer $d_B$, represented in base $B$, to its rightmost figure. 
In analogy with Lemma \ref{uprop}, the following properties hold for the map $u_B$.
\begin{lemma}\label{upropB}
Let $a_B$ and $b_B$ be two positive integers represented in base $B$. Then:
\begin{itemize}
\item $u_B (a_B+b_B) = u_B (u_B(a_B)+u_B(b_B))$;
\item $u_B (a_B \cdot b_B) = u_B (u_B(a_B) \cdot u_B(b_B))$.
\end{itemize}
\end{lemma}
\begin{proof}
The proof is verbatim the same as the proof of Lemma \ref{uprop} once one replaces all occurrences of $10, a, b$ by $B, a_B$ and $b_B$ correspondingly. 
\end{proof}

We can finally restate Theorem \ref{theo_proc} in the current more general setting. The proof is as in Theorem \ref{theo_proc}, once one replaces all the occurrences of $10$ and $u$ by $B$ and $u_B$ respectively.

\begin{theorem}\label{theo_proc_B}
Let $k$ and $n$ be two positive integers such that $1 \leq k < Bn-1$ and suppose that $ \frac{k}{Bn-1} = 0. \overline{a_1 \dots a_l}_B$, where $l$ is the length of the period of the base-$B$ representation $0. \overline{a_1 \dots a_l}_B$. 

For any $i \in \{1, \dots, l \}$ define the positive integers $\tilde{q}_i, \tilde{r}_i$ such that
\begin{displaymath}
\begin{cases}
B^i \cdot k = \tilde{q}_i (Bn-1) + \tilde{r}_i\\
0 \leq \tilde{r}_i < Bn-1.
\end{cases}
\end{displaymath}
Then, for any $i \in  \{1, \dots, l \}$:
\begin{itemize}
\item $u_B ((\tilde{q}_i)_B) = a_i$;
\item $u_B ((\tilde{q}_i)_B) = u_B ((\tilde{r}_i)_B)$.
\end{itemize} 
\end{theorem}

Finally, we consider a generic fraction $\frac{k}{m}$, where $k$ and $m$ are positive integers such that $1 \leq k < m$. Let
\begin{equation*}
B = p_1^{b_1} \dots p_f^{b_f}
\end{equation*}
be the factorization of $B$ in primes of $\N$, for a set $\{p_1, \dots, p_f \}$ of primes of $\N$ and a set $\{b_1, \dots, b_f \} \subseteq \N$. 

Suppose that
\begin{equation*}
m = p_1^{e_1} \dots p_f^{e_f} \cdot m',
\end{equation*}
where:
\begin{itemize}
\item $\{e_1, \dots, e_f \} \subseteq \N$;
\item $m' \in \N$;
\item $\gcd (m', B) = 1$.
\end{itemize}

Let $e$ be a positive integer such that $e \cdot b_i \geq e_i$ for any $i \in \{1, \dots, f \}$. Then,
\begin{equation*}
\frac{k}{m} = \frac{k \cdot p_1^{e b_1 -e_1} \dots p_f^{e b_f-e_f}}{B^e \cdot m'} = \frac{1}{B^{e}} \cdot \frac{k \cdot p_1^{e b_1-e_1} \dots p_f^{e b_f-e_f}}{m'}.
\end{equation*}
If necessary we divide $k \cdot p_1^{e b_1-e_1} \dots p_f^{e b_f-e_f}$ by $m'$. Hence, we can write
\begin{equation*}
\frac{k}{m} = \frac{1}{B^{e}} \cdot \left( k' + \frac{k''}{m'} \right),
\end{equation*}
where $\{k', k'' \} \subseteq \N$ and $1 \leq k'' < m''$.

Since $\gcd (B, m') = 1$, there exist two integers $c$ and $n$ such that 
\begin{equation*}
c m' = B n -1.
\end{equation*}
Moreover, we can choose $c$ and $n$ in such a way that 
\begin{equation*}
1 \leq c < B \text{ and } n > 0.
\end{equation*}
All considered, we get
\begin{equation*}
\frac{k}{m} = \frac{1}{B^{e}} \cdot \left( k' + \frac{c k''}{B n - 1} \right).
\end{equation*}
Summing the base-$B$ representation of the integer $k'$ to  the base-$B$ representation of $\frac{c k''}{B n - 1}$ and shifting the dot by $e$ positions to the left we get the base-$B$ representation of $\frac{k}{m}$.

\begin{example}
In this example we aim at computing the figures of the period of the base-$B$ expansion of $\frac{7}{20}$ for $B=12$. 

First, we notice that $12 = 2^2 \cdot 3$. Therefore, we can also write
\begin{equation*}
\frac{7}{20} = \frac{1}{4} \cdot \frac{7}{5} = \frac{3}{12} \cdot \frac{7}{5} = \frac{1}{B} \cdot \left( 4 + \frac{1}{5} \right).
\end{equation*}
Now, we concentrate on $\frac{1}{5}$. Denote  $m' = 5$. Then,
\begin{equation*}
c m' = B n -1
\end{equation*}
for $c = 7$ and $n = 3$. Hence, we can retrieve the figures of the period of the base-$B$ expansion of
\begin{equation*}
\frac{1}{5} = \frac{7}{35}
\end{equation*}
from the graph $G_{12,3}$, which is here represented (notice that the vertices are labelled using the base-$12$ representation, employing the figures from $0$ to $9$ and the letters $a, b$).
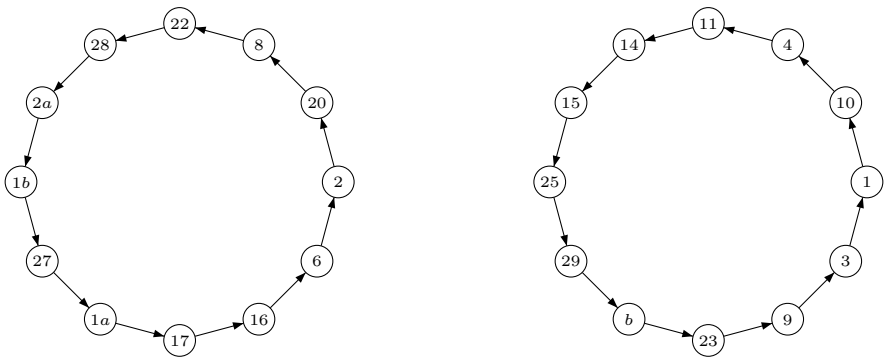
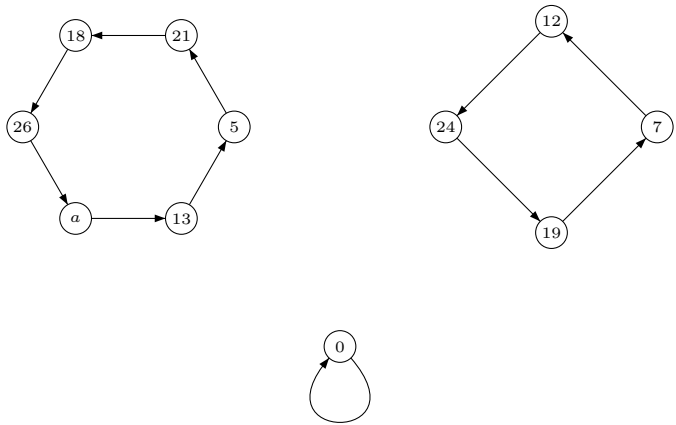
\begin{center}
    \unitlength=2pt
    \begin{picture}(160, 80)(-80,-40)
    \gasset{Nw=6,Nh=6,Nmr=3.0,curvedepth=0}
    \thinlines
    \tiny
  
    \node(M1)(80,0){$1$}
    \node(M2)(75.98,15){$10$}
    \node(M3)(65,25.98){$4$}
    \node(M4)(50,30){$11$}
    \node(M5)(35,25.98){$14$}
    \node(M6)(24.02,15){$15$}
    \node(M7)(20,0){$25$}
    \node(M8)(24.02,-15){$29$}
	\node(M9)(35,-25.98){$b$} 
	\node(M10)(50,-30){$23$}
	\node(M11)(65,-25.98){$9$}
	\node(M12)(75.98,-15){$3$}
	
    \node(N1)(-20,0){$2$}
    \node(N2)(-24.02,15){$20$}
    \node(N3)(-35,25.98){$8$}
    \node(N4)(-50,30){$22$}
    \node(N5)(-65,25.98){$28$}
    \node(N6)(-75.98,15){$2a$}
    \node(N7)(-80,0){$1b$}
    \node(N8)(-75.98,-15){$27$}
	\node(N9)(-65,-25.98){$1a$} 
	\node(N10)(-50,-30){$17$}
	\node(N11)(-35,-25.98){$16$}
	\node(N12)(-24.02,-15){$6$}
        
    \drawedge(M1,M2){}
    \drawedge(M2,M3){}
    \drawedge(M3,M4){}
    \drawedge(M4,M5){}
    \drawedge(M5,M6){}   
    \drawedge(M6,M7){}
	\drawedge(M7,M8){}
	\drawedge(M8,M9){}
	\drawedge(M9,M10){}
	\drawedge(M10,M11){}
	\drawedge(M11,M12){}
	\drawedge(M12,M1){}
	    
    \drawedge(N1,N2){}
    \drawedge(N2,N3){}
    \drawedge(N3,N4){}
    \drawedge(N4,N5){}
    \drawedge(N5,N6){}   
    \drawedge(N6,N7){}
	\drawedge(N7,N8){}
	\drawedge(N8,N9){}
	\drawedge(N9,N10){}
	\drawedge(N10,N11){}
	\drawedge(N11,N12){}
	\drawedge(N12,N1){}
\end{picture}
\end{center}

\begin{center}
    \unitlength=2pt
    \begin{picture}(160, 80)(-80,-40)
    \gasset{Nw=6,Nh=6,Nmr=3.0,curvedepth=0}
    \thinlines
    \tiny
  
    \node(M1)(-20,0){$5$}
    \node(M2)(-30,17.32){$21$}
    \node(M3)(-50,17.32){$18$}
    \node(M4)(-60,0){$26$}
    \node(M5)(-50,-17.32){$a$}
    \node(M6)(-30,-17.32){$13$}

    \node(N1)(60,0){$7$}
    \node(N2)(40,20){$12$}
    \node(N3)(20,0){$24$}
    \node(N4)(40,-20){$19$}
        
    \drawedge(M1,M2){}
    \drawedge(M2,M3){}
    \drawedge(M3,M4){}
    \drawedge(M4,M5){}
    \drawedge(M5,M6){}   
    \drawedge(M6,M1){}
	    
    \drawedge(N1,N2){}
    \drawedge(N2,N3){}
    \drawedge(N3,N4){}
    \drawedge(N4,N1){}
\end{picture}
\end{center}

\begin{center}
    \unitlength=2pt
    \begin{picture}(120, 15)(-60,-15)
    \gasset{Nw=6,Nh=6,Nmr=3.0,curvedepth=0}
    \thinlines
    \tiny
  
    \node(N1)(0,0){$0$}
                 
	\drawloop[loopangle=-90](N1){}
	\end{picture}
\end{center}

We notice that $7_{12}$ belongs to a cycle of length $4$. Therefore, the length of the period of the base-$12$ representation of $\frac{1}{5} = \frac{7}{35}$ is $4$ and
\begin{equation*}
\frac{1}{5} = 0.\overline{2497}_{12}.
\end{equation*} 
Hence,
\begin{equation*}
\frac{7}{20} =  \frac{1}{12} \cdot \left( 4 + \frac{1}{5} \right) = 0.4 \overline{2497}_{12}. 
\end{equation*}
\end{example} 

\bibliography{Refs}
\end{document}